\documentclass[10pt,reqno]{amsart}
\usepackage{amsfonts}
\usepackage{amsmath}
\usepackage{CJK}
\usepackage{amssymb}
\usepackage{latexsym,bm}
\usepackage[RGB]{xcolor}
\usepackage{mathrsfs}
\usepackage{amsthm}
\numberwithin{equation}{section}
\newtheorem{theorem}{Theorem}[section]
\newtheorem{assumption}[theorem]{Assumption}
\newtheorem{lemma}[theorem]{Lemma}

\newtheorem{proposition}[theorem]{Proposition}
\newcommand{\dif}{\mathrm{d}}

\newcommand{\SUM}[3]{\sum\limits_{{#1}={#2}}^{#3}}

\begin{document}
\title[IBVP for relaxed CNS]{Initial boundary value problem for one-dimensional hyperbolic compressible Navier-Stokes equations}
\author{Yuxi Hu and Yachun Li}
 \thanks{\noindent Yuxi Hu, Department of Mathematics, China University of Mining and Technology, Beijing, 100083, yxhu86@163.com\\
Yachun Li, School of Mathematical Sciences, CMA-Shanghai, MOE-LSC, and SHL-MAC, Shanghai Jiao Tong University,  Shanghai, China, ycli@sjtu.edu.cn
}
\begin{abstract}
An initial boundary value problem for one-dimensional hyperbolic compressible Navier-Stokes equations is investigated. After transforming the system into Lagrangian coordinate, the resulting system possesses a structure with uniform characteristic boundary. By constructing an approximate system with non-characteristic boundary, we  get a uniform  global smooth solutions and obtain a global solution of the original problem by passing to a limit. Moreover, the global  relaxation limit  is also obtained.  \\
{\bf Keywords}: Initial boundary value problem; hyperbolic  compressible Navier-Stokes equations; global solution\\ 
 {\bf AMS classification code}: 35 L 60, 35 B 44
\end{abstract}
\maketitle
\section{Introduction}
In this paper, we investigate the global well-posedness of solutions for one-dimensional relaxed compressible Navier-Stokes equations in a bounded domain. The equations are described as follows
\begin{align}\label{1.1}
\begin{cases}
\rho_t+(\rho u)_x=0,\\
(\rho u)_t+(\rho u^2)_x+p(\rho)_x=S_x,\\
\tau \rho (S_t+ u S_x)+S=\mu u_x,
\end{cases}
\end{align}
where $(t,x)\in (0, \infty) \times \Omega$ and $\Omega$ is a bounded domain in $\mathbb R$. Here, $\rho$, $u$, $S$ denote fluid density, velocity and stress, respectively. 
The pressure $p$ is assumed to satisfy the usual $\gamma$-law, $p(\rho)=a \rho^\gamma$.

The equations $\eqref{1.1}_1$ and $\eqref{1.1}_2$ are consequences of conservation of mass and momentum, respectively. The constitutive equation $\eqref{1.1}_3$ in multi-d form  is proposed recently by Freist\"uhler \cite{Frei1}, going back to Ruggeri \cite{Rug83} and M\"uller\cite{Mu67}, by modifying the usual Maxwell's law \cite{Max1867} with $\rho$ in front of the material derivative of $S$. Such modification make the system into a conservation form and allow one to define weak solutions. Furthermore, system \eqref{1.1} implies an energy dissipation equation
\begin{align}\label{1.2}
\left( \rho h(\rho)+\frac{1}{2} \rho u^2+\frac{\tau}{2\mu} \rho S^2 \right)_t+\left( \rho u h(\rho)+\frac{1}{2} \rho u^3 +\frac{\tau}{2\mu}\rho u S^2\right)_x+\frac{1}{\mu} S^2=0,
\end{align}
where $h^\prime(\rho)=\frac{p(\rho)}{\rho^2}$. The equation \eqref{1.2} allow one to get $L^2$ estimates of solutions immediately. 

We further remark that  in the constitutive relation $\eqref{1.1}_3$, in its linearized form: $\tau S_t+S=\mu u_x$, the positive parameter $\tau$ denotes the relaxation time describing the time lag in the response of the stress tensor to the velocity gradient, cf. also Christov and Jordan \cite{CJo}.   Pelton et al. \cite{Peetal} showed that such a"time lag'' cannot be neglected, even for simple fluids, in
the experiments of high-frequency vibration of nano-scale mechanical devices immersed in water-glycerol mixtures. It turned out
  that, cf. also \cite{ChSa015}, equation $\eqref{1.1}_3$ provides a general formalism  to characterize the fluid-structure interaction
of nano-scale mechanical devices vibrating in simple fluids.

We are interested in the initial boundary value problem to system \eqref{1.1}   for the functions
\begin{align*}
(\rho, u, S): [0, +\infty) \times \Omega \rightarrow (0, \infty)\times \mathbb R \times \mathbb R
\end{align*}
with initial conditions
\begin{align} \label{1.3}
(\rho, u, S)(0,x)=(\rho_0, u_0, S_0)(x)
\end{align}
 and boundary conditions
 \begin{align}\label{1.4}
u\big|_{\partial\Omega}=0.
\end{align}

We note that the Cauchy problem for system \eqref{1.1} and its multi-d form have been widely studied. In particular, Yong \cite{YW14} first studied the isentropic Navier-Stokes equations with revised Maxwell law( in its linear form) and get a local well-posedness theory and local relaxation limit. This results are extend by Hu and Racke \cite{HR2017} to non-isentropic case and by Peng \cite{PY2020} to more general case. The blow-up phenomenon was studied by Hu and Wang \cite{HW2019, HW2020} and B\"arlin \cite{BJ2022}. By neglecting $\rho$ in equation $\eqref{1.1}_3$, the asymptotic stability of viscous shock wave and rarefaction wave were also investigated, see \cite{HW2022, HW2023, Frei2} for more details. However, to our knowledge, there is few results concerning the initial boundary value problems. The aim of this manuscript is to study the boundary effect to the relaxed system and whether the local and/or global well-posedness can be obtained with the boundary.

We rewrite the system \eqref{1.1} in Lagrangian coordinates as 
\begin{align}\label{1.5}
\begin{cases}
v_t=u_x,\\
u_t+p(v)_x=S_x,\\
\tau S_t +vS=\mu u_x,
\end{cases}
\end{align}
where $v=\frac{1}{\rho}$ is the specific volume per unit mass. The initial and boundary conditions \eqref{1.3}-\eqref{1.4} turn to 
\begin{align}\label{1.6}
(v, u, S)(0,x)=(v_0, u_0, S_0)(x)
\end{align}
and
\begin{align}\label{1.7}
u\big|_{\partial\Omega}=0.
\end{align}

Let $U=(v,u,S)$, then we have
\begin{align} \label{1.8}
A^0U_t+A^1(U) U_x+B(U)U=0,
\end{align}
where 
\begin{align*}
A^0=\mathrm{diag}\{1, 1, \tau\}, A^1(U)=
\begin{pmatrix}
0&-1&0\\
p^\prime(v)&0&-1\\
0&-\mu&0
\end{pmatrix},
B(U)=\mathrm{diag}\{0, 0, v\}.
\end{align*}
Note that $\det{ (A^0)^{-1}A^1(U)}=0$ for any $U \in G:=\{ (0, \infty)\times \mathbb R \times \mathbb R\}$. Therefore, we get a problem with uniform characteristic boundary. 
For such problem, there is no general well-posedness result, even locally, \cite{SCH86, Chen2007}.  
The difficulty lies in that one can not get the estimates of $U_x$ or $U_{xx}$ directly by the estimates of $U_t$ and $U_{tt}$. Therefore, the loss of derivative may occur near the boundary. 

Note that when $\tau=0$, the system \eqref{1.1}  reduces to the classical isentropic Navier-Stokes equations
\begin{align}\label{CNS}
\begin{cases}
v_t=u_x,\\
u_t+p(v)_x=\left(\frac{\mu u_x}{v}\right)_x,\\
\end{cases}
\end{align}
 for which the global large solution (away from vacuum) was  already known, see \cite{Kanel1968}.
 But the methods there can not be applied to the relaxed system due to the essential change of structure, i.e., from hyperbolic-parabolic to pure hyperbolic system. On the other hand, it has been show that, see \cite{HW2019, BJ2022},
 solutions to the relaxed system may blowup in finite time for some large data. Therefore, a global defined  smooth solutions should  not be expected for large data.

Our approach is as follows. First, following the ideas in \cite{SCH86} for compressible Euler equations, we construct an approximating system with non-characteristic boundary.  By justifying that the boundary condition is maximally nonnegative, we give a local well-posedness theory by use of classical result. Second, we establish a uniform a priori estimates and get a uniform global smooth solutions for the approximating system. In this step,  some high-order boundary terms  appear when one do integration by part. This problem is solved   by using of the essential structure of the system and by introducing some auxiliary functions. Finally, by usual compactness-arguement, we get a global smooth solution for original system \eqref{1.5}-\eqref{1.7}.

The following assumptions are needed throughout the paper:
\begin{assumption}\label{assu}
\mbox{}\\
 (1) The initial and boundary data satisfy the usual compatibility condition
\begin{align}\label{comp}
	\partial^k_tu(0,x)\big|_{\partial\Omega}=0,
	k=0,1
\end{align}
where $\partial_tu(0,x)$  are defined recursively by equation $\eqref{1.5}_2$. \\
(3) The initial data is well-prepared in the following sense:
\begin{align}\label{new-hu4-1}
 \| v_0S_0-\mu (u_0)_x \|_{H^1} =O(\sqrt{\tau}),\, \mathrm{as}\, \tau \rightarrow 0.
\end{align}
(4) $V_0^k \in H^{2-k} $ for $k=0, 1, 2$, where 
\begin{align*}
V_0^k:= \partial_t^k\left(v-1, u , \sqrt{\tau} S\right)(t=0, \cdot), \, k=0, 1\\
V_0^2:= \tau \partial_t^2 (v-1, u,  \sqrt{\tau} S)(t=0, \cdot),
\end{align*}
where $V_0^1, V_0^2$ are defined recursively by equations \eqref{1.5}. 
\end{assumption}

Our main results are stated as follows.
\begin{theorem}\label{th1.1}
Let Assumption \ref{assu} hold. Then, there exists a constant $\epsilon_0>0$ such that if 
\begin{align}
E_0:=\SUM k 0 2 \|V_0^k\|_{H^{2-k}}^2<\epsilon_0
\end{align}
the system \eqref{1.5}-\eqref{1.7} has a unique global solution $(v, u, S)\in C^1([0,\infty)\times \Omega)$ with
\begin{align}
(v-1, u, S)\in C([0, \infty), H^{2-\delta_0}(\Omega)) \cap C^1([0, \infty), H^{1-\delta_0}(\Omega))
\end{align}
for any $\delta_0>0$ and
satisfying
\begin{align}
\sup_{0\le t <\infty}\|(v-1, u, \sqrt{\tau}S)(t, \cdot)\|_{H^2}^2+ \int_0^\infty \left( \|(v_x, u_x)\|_{H^1}^2+ \|S\|_{H^2}^2 \right) \dif t \le C E_0,
\end{align}
where $C$ is a universal constant independent of $\tau$.
\end{theorem}
Based on the uniform estimates of solutions, we have the following convergence theorem.
\begin{theorem}\label{th1.2}
(Global weak convergence). Let  $(v^\tau, u^\tau, S^\tau)$ be the global solutions obtained in Theorem \ref{th1.1}, then there exists
functions $( v^0,  u^0)\in L^\infty(\mathbb R^+; H^2(\Omega))$ and $S^0\in L^2(\mathbb R^+; H^2(\Omega))$, such that, as $\tau\rightarrow 0$ up to subsequences,
\begin{align}
(v^\tau, u^\tau) \rightharpoonup ( v^0,  u^0) \qquad \mathrm{weak}-\ast \quad \mathrm{in}\quad L^\infty(\mathbb R^+; H^2(\Omega)),\label{1.17}\\
S^\tau \rightharpoonup  S^0 \qquad \mathrm{weakly} \quad \mathrm{in}\quad L^2(\mathbb R^+; H^2(\Omega)), \label{1.18}
\end{align}
where $(v^0,  u^0)$ is the solution to the one-dimensional  isentropic compressible Navier-Stokes equations  \eqref{CNS}, with initial value $( v_0^0,  u_0^0)$ which 
is the weak limit of $(v_0^\tau, u_0^\tau)$ up to subsequences. Moreover, 
\begin{align*}
 S^0=\mu   \frac{(u^0)_x}{v^0}.
\end{align*}
\end{theorem}

The paper is organized as follows. In Section 2 we show the local existence theorem of an approximated system. The uniform a priori estimates is established in Section 3. In Section 4, we justify the limit and prove our main theorem.

Finally, we introduce some notation.  $W^{m,p}=W^{m,p}( \Omega),\,0\le m\le
\infty,\,1 \le p\le \infty$, denotes the usual Sobolev space with norm $\|
\cdot \|_{W^{m,p}}$, $H^m$ and $L^p$
stand for $W^{m,2}$ resp. $W^{0,p}$. Without loss of generality, we take $\Omega=[0, 1]$.

\section{Approximated system and Local existence}
In this part, we construct an approximate system with non-characteristic boundary, for which a unique local defined smooth solution is given.


Our approximate system is as 
follows:
\begin{align}\label{2.1}
\begin{cases}
v_t^\epsilon=u_x^\epsilon,\\
u_t^\epsilon+p(v^\epsilon)_x=S_x^\epsilon,\\
\tau (S_t^\epsilon+\epsilon b(x) S_x^\epsilon)+vS^\epsilon=\mu u_x^\epsilon, 
\end{cases}
\end{align}
where $b(x)=2x-1$.
Initial and boundary conditions are given by 
\begin{align} \label{2.2}
(v^\epsilon,u^\epsilon,S^\epsilon)(0,x)=(v_0, u_0, S_0)(x)
\end{align}
and
\begin{align}\label{2.3}
u^\epsilon(t,0)=u^\epsilon(t,1)=0.
\end{align}
Let $U^\epsilon=(v^\epsilon,u^\epsilon,S^\epsilon)$, then we have
\begin{align} \label{2.4}
A^0U^\epsilon_t+A^1(U^\epsilon) U^\epsilon_x+B(U^\epsilon)U^\epsilon=0,
\end{align}
where 
\begin{align*}
A^0=\mathrm{diag}\{-p^\prime(v^\epsilon), 1, \frac{\tau}{\mu}\}, A^1(U^\epsilon)=
\begin{pmatrix}
0&p^\prime(v^\epsilon)&0\\
p^\prime(v^\epsilon)&0&-1\\
0&-1&\frac{\tau}{\mu} \epsilon b(x)
\end{pmatrix},
B(U^\epsilon)=\mathrm{diag}\{0, 0, \frac{v^\epsilon}{\mu}\}
\end{align*}
with initial condition $U^\epsilon (0,x)=U_0:=(v_0, u_0, S_0)(x)$ and boundary condition
\begin{align}\label{boundary}
MU^\epsilon|_{\partial \Omega}=0, \quad\mathrm{with}\quad
M=
\begin{pmatrix}
0&0&0\\
0&1&0\\
0&0&0
\end{pmatrix}.
\end{align}
Note that $\det{ (A^0)^{-1}A^1(U^\epsilon)}\big|_\Omega=p^\prime(v^\epsilon) \epsilon b(x)|_\Omega\neq0$ for any $U \in G:=\{ (0, \infty)\times \mathbb R \times \mathbb R\}$. 
So, the boundary condition \eqref{2.3} is a  non-characteristic boundary for any $\epsilon>0$. 

Now, we show the boundary condition is maximally nonnegative, i.e., the matrix $A^1(U^\epsilon) \cdot \nu\big|_{\partial \Omega}$ is positive semidefinite on the null space  $N$ of $M$ but not on any space containing $N$,  see \cite{SCH86}. Let $\xi=(\xi_1, 0, \xi_2)^T\in \ker M=\mathrm{span} \{(1,0,0)^T, (0,0,1)^T\}$, then 
$$
\xi^T A^1\cdot \nu \big|_{\partial\Omega}\xi=\frac{\tau \epsilon}{\mu} \xi_2^2 \ge 0.
$$
On the other hand, the only space containing $\ker M$  as a proper subspace is $\mathbb R^3$. So, we take $p=(1,1,0)^T\in \mathbb R^3$, and calculate 
$$
p^T A^1\cdot \nu \big|_{\partial\Omega} p=2 p^\prime(v^\epsilon) < 0.
$$
Thus, the maximally nonnegative property is satisfied.
Therefore, classical results implies local  well-posedness theory, see \cite{SCH86}.  
\begin{theorem}
Suppose $(v_0, u_0, S_0)\in H^2$ satisfying the compatibility condition \eqref{comp} and $$\min_{x\in[0,1]} v_0(x)>0.$$ Then there exists a unique local solution $(v^\epsilon, u^\epsilon, S^\epsilon)$ to initial boundary value problem \eqref{2.1}-\eqref{2.3} on some time interval $[0, T]$ with
\begin{align*}
(v^\epsilon, u^\epsilon, S^\epsilon)\in C^0([0,T], H^2)\cap C^1([0, T], H^1),\\
\min_{x\in[0,1]} v(t, x)>0,\qquad \forall t >0.
\end{align*}
\end{theorem}

\section{Uniform a priori estimates}
In this part, we shall get a uniform (with respect to $\tau$ and $\epsilon$) a priori estimates  which allows us to get uniform global solutions. 
For simplicity, we still denote $(v^\epsilon, u^\epsilon, S^\epsilon)$ by $(v, u, S)$ without confusion. 

First, we define the energy 
\begin{align} \label{3.1}
E(t):=\sup_{0\leq s \leq t} \left(\SUM k 0 1 \left\|\partial_t^k(v-1, u, \sqrt{\tau  }S) (s, \cdot)\right\|_{H^{2-k}}^{2}+\tau^2 \| \partial_t^2 (v, u,  \sqrt{\tau} S)(s, \cdot)\|_{L^2}^2\right)
\end{align}
and dissipation
\begin{align}\label{3.2}
\mathcal{D}(t):=\sum_{|\alpha|=1}^{2}\left\|D^{\alpha}(v,u) (t, \cdot)\right\|_{L^{2}}^{2}+ \SUM k 0 1 \left\|\partial_t^k(S)(t, \cdot)\right\|_{H^{2-k}}^{2}+\tau^2 \| S_{tt}(t, \cdot)\|_{L^2}^2
\end{align}
where $D=(\partial_t, \partial_x)$.  We are aiming to show the following a priori estimates.
\begin{proposition}\label{Prop3.1}
Let $(v, u, S)\in C^0([0,T], H^2)$ be local solutions to system \eqref{2.1}-\eqref{2.3}.  Assume that there exists  a small $\delta$ such that $E(t)\le \delta$, then  we have
\begin{align}\label{3.3}
E(t)+\int_0^t \mathcal D(s)\dif s \le C \left(E_0+E^\frac{1}{2}(t) \int_0^t \mathcal D(s)\dif s\right),
\end{align}
where $C$ is a constant independent of $\epsilon$ and $\tau$.
\end{proposition}
In the following lemmas, we  assume $\delta$ is small enough such that
\begin{align}\label{3.4}
\frac{3}{4}\le \sup_{(t, x)\in[0,T]\times[0,1]}v(t,x) \le \frac{5}{4}.
\end{align}
Moreover, without loss of generality, we assume $\tau\le 1$ and $\epsilon<<1$ . $C$ denotes a universal constant which is independent of $\tau$ and $\epsilon$. 

First, we have the following $L^2$-estimates of solutions.
\begin{lemma}\label{Le3.2}
There exists some constant $C$ such that 
\begin{align}\label{3.5}
\frac{\dif}{\dif t} \int_0^1 \left( (v-1-h(v))+\frac{1}{2}u^2+\frac{\tau}{2\mu} S^2 \right) \dif x +\frac{1}{4\mu}\int_0^1  S^2 \dif x \le 0,
\end{align}
where $h(v)=\frac{1}{1-\gamma}(v^{1-\gamma}-1)$ satisfying $h^\prime(v)=p(v)$ and $h(1)=0$.
\end{lemma}
\begin{proof}
Multiplying the equations $\eqref{2.1}_2$ and $\eqref{2.1}_3$ by $u$ and $\frac{1}{\mu}S$, respectively,  and using the boundary condition \eqref{2.3}, we get the following equality,
\begin{align*}
\frac{\dif}{\dif t} \int_0^1 \left(\frac{1}{2}u^2+\frac{\tau}{2\mu} S^2-h(v)\right)\dif x +\epsilon \int_0^1 \frac{b(x)}{\mu} \left(\frac{1}{2}S^2\right)_x \dif x +\int_0^1 \frac{v}{\mu} S^2\dif x 
=0.
\end{align*}
Note that $b(0)=-1, b(1)=1$, we have
\begin{align*}
\int_0^1 \epsilon \frac{b(x)}{\mu} \left(\frac{1}{2}S^2\right)_x \dif x
=\epsilon \frac{b(x)}{2\mu} S^2\big|_0^1-\frac{\epsilon}{\mu} \int_0^1 S^2 \dif x \ge -\frac{\epsilon}{\mu} \int_0^1 S^2 \dif x.
\end{align*}
Choosing $\epsilon$ such that $\epsilon \le \frac{1}{4}$ and combining the mass equation $\eqref{2.1}_1$, we get
\begin{align*}
\frac{\dif}{\dif t} \int_0^1 \left( \frac{1}{2} u^2+\frac{\tau}{2\mu}S^2+v-1-h(v)\right)\dif x +\frac{1}{4\mu} \int_0^1 S^2 \dif x \le 0.
\end{align*}
This finish the proof of  Lemma \ref{Le3.2}.
\end{proof}
Now, we do the first-order estimates.  Taking derivative with respect to $x$ to the equations \eqref{2.1}, we get
\begin{align}\label{3.6}
\begin{cases}
v_{tx}=u_{xx},\\
u_{tx}+p(v)_{xx}=S_{xx},\\
\tau S_{tx}+\tau\epsilon (b(x) S_x)_x+(vS)_x=\mu u_{xx}.
\end{cases}
\end{align}
Multiplying the equation $\eqref{3.6}_2$by $u_x$ and integrating over $[0, 1]$ with respect to $x$, we get
\begin{align*}
\frac{\dif}{\dif t} \int_0^1 \frac{1}{2} u_x^2 \dif x +\int_0^1 (p(v)_{xx}-S_{xx})u_x \dif x=0.
\end{align*}
Note that $u_t(t,0)=u_t(t,1)=0$, by the momentum equation $\eqref{2.1}_2$, the term $p(v)_x-S_x$ vanishes on both sides of the boundary. Therefore, we have
\begin{align*} 
&\int_0^1 (p(v)_{xx}-S_{xx})u_x\dif x \\
&=(p(v)_x-S_x)u_x\big|_0^1-\int_0^1 (p(v)_x-S_x)u_{xx} \dif x\\
&=-\int_0^1 p(v)_x v_{tx}\dif x+\int_0^1 S_x u_{xx}\dif x\\
&=-\frac{\dif}{\dif t} \int_0^1 \frac{1}{2}p^\prime(v) v_x^2 \dif x+\frac{1}{2}\int_0^1 p^{\prime\prime}(v)v_t  v_x^2\dif x++\int_0^1 S_x u_{xx}\dif x.
\end{align*}
Multiplying the  equation  $\eqref{3.6}_3$ by $\frac{1}{\mu}S_x$ and integrating the result, we get
\begin{align*}
\frac{\dif}{\dif t} \int_0^1\frac{\tau}{2\mu} S_x^2\dif x+\frac{\tau\epsilon}{\mu} \int_0^1 (b(x)S_x)_x S_x\dif x+\frac{1}{\mu}\int_0^1 (vS)_x S_x \dif x =\int_0^1 u_{xx}S_x\dif x.
\end{align*}
Note that $b^\prime(x)=2>0$, we have
\begin{align*}
\frac{\tau\epsilon}{\mu} \int_0^1 (b(x)S_x)_x S_x\dif x
=\frac{\tau\epsilon}{\mu} \int_0^1 \left( b^\prime S_x^2+b\left(\frac{1}{2} S_x^2 \right)_x\right)\dif x
\ge -\frac{\tau\epsilon}{\mu} \int_0^1 S_x^2 \dif x
\end{align*}
and
\begin{align*}
\frac{1}{\mu} \int_0^1 (vS)_x S_x \dif x \ge \frac{3}{4\mu} \int_0^1 S_x^2 \dif x -E^\frac{1}{2}(t) \mathcal D(t).
\end{align*}
Combining the above estimates, we derive the following lemma.
\begin{lemma}\label{Le3.3}
There exists a constant $C$ such that
\begin{align}\label{3.7}
\frac{\dif}{\dif t} \int_0^1 \left( \frac{1}{2} u_x^2-\frac{1}{2}p^\prime(v) v_x^2+\frac{\tau}{2\mu} S_x^2 \right)\dif x +\frac{1}{4\mu}\int_0^1 S_x^2 \dif x 
\le C E^\frac{1}{2}(t) \mathcal D(t).
\end{align}
\end{lemma}
In quite a similar way, we can show first order estimates with respect to $t$ as follows.
\begin{lemma}\label{Le3.4}
There exists  a constant $C$ such that
\begin{align}\label{3.8}
\frac{\dif}{\dif t} \int_0^1 \left( \frac{1}{2} u_t^2-\frac{1}{2}p^\prime(v) v_t^2+\frac{\tau}{2\mu} S_t^2 \right)\dif x +\frac{1}{4\mu}\int_0^1 S_t^2 \dif x 
\le C E^\frac{1}{2}(t) \mathcal D(t) .
\end{align}
\end{lemma}
Combining the above lemmas, we conclude that 
\begin{lemma}\label{Le3.5}
There exists a constant $C$ such that
\begin{align}\label{3.9}
&\SUM \alpha 0 1  \| D^\alpha(v-1, u, \sqrt{\tau} S\|_{L^2}^2 +\int_0^t \SUM \alpha 0 1 \|D^\alpha  S\|_{L^2}^2 \dif t \le C(E_0+E^\frac{1}{2}(t) \int_0^t \mathcal D (s)\dif s).
\end{align}
\end{lemma}
The next lemma show the dissipative estimates of $D(u, v)$.
\begin{lemma}\label{Le3.6}
There exists a constant $C$ such that
\begin{align}\label{3.10}
\int_0^t \int_0^1|D(u,v)|^2 \dif x \dif t \le C(E_0+E^\frac{1}{2}(t) \int_0^t \mathcal D (s)\dif s).
\end{align}
\end{lemma}
\begin{proof}
From the equations $\eqref{2.1}_3$ and $\eqref{2.1}_1$,  using Lemma \ref{Le3.5}, we get immediately 
\begin{align}\label{3.11}
\int_0^t \int_0^1 (u_x^2+v_t^2)\dif x \dif t \le C(E_0+E^\frac{1}{2}(t) \int_0^t \mathcal D (s)\dif s).
\end{align}
On the other hand, by multiplying the equation  $\eqref{2.1}_2$ by $u_t$, it yields
\begin{align*}
&\int_0^t \int_0^1 u_t^2 \dif x \dif t =-\int_0^t\int_0^1 p_x(v) u_t \dif x\dif t +\int_0^t\int_0^1 S_x u_t \dif x \dif t\\
&=\int_0^t \frac{\dif}{\dif t} \int_0^1 p_x(v) u\dif x \dif t +\int_0^t \int_0^1p_{t}(v) u_x\dif x\dif t+\int_0^t\int_0^1 S_x u_t \dif x \dif t\\
&\le \frac{1}{2} \int_0^t \int_0^1 u_t^2 \dif x \dif t + C\int_0^1 (v_x^2+u^2)\dif x +C\int_0^t\int_0^1( v_t^2+u_x^2+S_x^2)\dif x \dif t
\end{align*}
which gives
\begin{align}\label{3.12}
\int_0^t \int_0^1 u_t^2 \dif x \dif t \le C(E_0+E^\frac{1}{2}(t) \int_0^t \mathcal D (s)\dif s).
\end{align}
By equation $\eqref{2.1}_2$, one also get
\begin{align}\label{3.13}
\int_0^t \int_0^1 v_x^2 \dif x \dif t \le C(E_0+E^\frac{1}{2}(t) \int_0^t \mathcal D (s)\dif s).
\end{align}
Thus, \eqref{3.11}-\eqref{3.13} implies \eqref{3.10} immediately. 
\end{proof}
Now, we are ready to give the second derivative estimates of solutions. 
\begin{lemma} \label{Le3.7}
There exists a constant $C$ such that
\begin{align}
&\int_0^1 \tau^2(-\frac{1}{2} p^\prime(v)v_{tt}^2+\frac{1}{2}u_{tt}^2+\frac{\tau}{2\mu}S_{tt}^2 )\dif x +\frac{\tau^2}{4\mu}\int_0^t\int_0^1 S_{tt}^2 \dif x \le C\left(E_0+ E^\frac{1}{2}(t) \int_0^t \mathcal D(s)\dif s \right). \label{3.14}
\end{align}
\end{lemma}
\begin{proof}
Dividing the the equation $\eqref{2.1}_3$ by $v$, and then taking derivative with respect to $t$ twice, we get
\begin{align}\label{3.15}
\begin{cases}
v_{ttt}=u_{xtt},\\
u_{ttt}+p(v)_{xtt}=S_{xtt},\\
\tau \left(\frac{S_t}{v}\right)_{tt}+\left(\tau \epsilon b(x)\frac{S_x}{v} \right)_{tt}+S_{tt}=\left(\frac{\mu}{v} u_x\right)_{tt}.
\end{cases}
\end{align}
Multiplying the  equations  $\eqref{3.15}_1$ and $\eqref{3.15}_2$by $u_{tt}$ and $-p^\prime(v)v_{tt}$, respectively, and integrating the result, we have
\begin{align}\label{3.16}
\frac{\dif}{\dif t} \int_0^1 \left(\frac{1}{2} u_{tt}^2-\frac{1}{2} p^\prime(v)v_{tt}^2 \right)\dif x+\int_0^1 S_{tt} u_{ttx} \dif x
=\int_0^1 \left(p(v)_{tt} u_{ttx} -p^\prime(v) v_{tt} u_{ttx}-\frac{1}{2} p^{\prime\prime}(v) v_t v_{tt}^2\right) \dif x,
\end{align}
where
\begin{align*}
&\int_0^1 \left(p(v)_{tt} u_{ttx} -p^\prime(v) v_{tt} u_{ttx} -\frac{1}{2} p^{\prime\prime}(v) v_t v_{tt}^2\right)\dif x=\int_0^1 \left(p^{\prime\prime}(v) v_t^2 u_{ttx}-\frac{1}{2} p^{\prime\prime}(v) v_t v_{tt}^2\right) \dif x\\
&=-\int_0^1 p^{\prime\prime\prime}(v)v_xv_t^2 u_{tt} \dif x-\int_0^1 2p^{\prime\prime}(v) v_t v_{tx} u_{tt}\dif x -\int_0^1 \frac{1}{2} p^{\prime\prime}(v) v_t v_{tt}^2\dif x \le C E^\frac{1}{2}(t) \mathcal D(t).
\end{align*}
Note that the above estimate (also in what follows) is formal because the regularity of the local solution is insufficient to validate several of the steps. In particular, we use the high-order compatibility condition $u_{tt}(t,0)=u_{tt}(t,1)=0$. However, a standard density argument will eliminate the needs for the extra regularity of local solutions (cf. \cite{SCH86}).

Next, multiplying the equation $\eqref{3.15}_3$ by $\frac{v}{\mu} S_{tt}$, and integrating the result, we get
\begin{align*}
\int_0^1 \tau \left(\frac{S_t}{v}\right)_{tt} \frac{v}{\mu} S_{tt} \dif x+\int_0^1\left( \tau\epsilon b(x) \frac{S_x}{v}\right)_{tt} \frac{v}{\mu} S_{tt}\dif x+\int_0^1 \frac{v}{\mu} S_{tt}^2 \dif x 
=\int_0^1 \left(\frac{\mu u_x}{v}\right)_{tt} \frac{v}{\mu} S_{tt} \dif x.
\end{align*}
We estimate each term of the above equation separately.  For the left-hand side of the above equation, we have
\begin{align*}
&\int_0^1 \tau \left(\frac{S_t}{v}\right)_{tt} \frac{v}{\mu} S_{tt} \dif x\\
&=\int_0^1 \tau \left( \frac{1}{v} S_{ttt}-\frac{2}{v^2} v_t S_{tt}+\frac{2}{v^3}v_t^2 S_t-\frac{1}{v^2} v_{tt} S_t\right) \frac{v}{\mu} S_{tt} \dif x \ge \frac{\dif}{\dif t} \int_0^1 \frac{\tau}{2\mu}  S_{tt}^2 \dif x-\frac{C}{\tau} E^\frac{1}{2}(t) \mathcal D(t)
\end{align*}
and
\begin{align*}
&\int_0^1 \tau \epsilon b(x) \left(\frac{S_x}{v}\right)_{tt} \frac{v}{\mu} S_{tt}\dif x\\
&=\int_0^1 \tau \epsilon b(x) \left( \frac{1}{v} S_{xtt}-\frac{2}{v^2} v_t S_{tx}+\frac{v_t^2}{v^3} S_x-\frac{v_{tt}}{v^2} S_x\right) \frac{v}{\mu} S_{tt} \dif x \\
&\ge \frac{1}{2\mu} \tau \epsilon b(x) S_{tt}^2\big|_0^1-\int_0^1 \frac{\tau \epsilon}{\mu} S_{tt}^2 \dif x-C E^\frac{1}{2}(t) \mathcal D(t).
\end{align*}
For the right-hand side of the above equation, we have
\begin{align*}
\int_0^1 \left(\frac{\mu}{v} u_x\right)_{tt} \frac{v}{\mu} S_{tt}\dif x
&=\int_0^1 \left(\frac{2v}{v^3} v_t^2u_x-\frac{1}{v^2} v_{tt}u_x-2 \frac{1}{v^2} v_t u_{xt}+\frac{1}{v} u_{xtt}\right) v S_{tt} \dif x\\
&\le \int_0^1 u_{xtt} S_{tt} \dif x +\frac{C}{\tau} E^\frac{1}{2}(t) \mathcal D(t).
\end{align*}

Therefore, we derive that
\begin{align}\label{3.17}
\frac{\dif}{\dif t} \int_0^1 \frac{\tau}{2\mu} S_{tt}^2 \dif x-\int_0^1 u_{xtt} S_{tt}\dif x+\frac{1}{4\mu} \int_0^1 S_{tt}^2\dif x \le \frac{C}{\tau} E^\frac{1}{2}(t) \mathcal D(t).
\end{align}
Combining \eqref{3.16} and \eqref{3.17}, integrating  the result over $(0, t)$ and noticing that 
\begin{align*}
\int_0^1 \tau^2 \left(  u_{tt}^2+  v_{tt}^2+ \tau S_{tt}^2 \right)    (t=0, x)\dif x
 \le C \|V_0^2\|_{L^2}^2 \le CE_0,
\end{align*}
we get the desired result.
\end{proof}


\begin{lemma}\label{Le3.8}
There exists a constant $C$ such that
\begin{align}
\int_0^1 \left(-\frac{1}{2} p^\prime(v)v_{tx}^2+\frac{1}{2}u_{tx}^2+\frac{\tau}{2\mu}S_{tx}^2 \right)\dif x +\frac{1}{4\mu}\int_0^t\int_0^1 S_{xt}^2 \dif x \nonumber \\ 
 \le C\left(E_0+E^\frac{3}{2}(t)+ E^\frac{1}{2}(t) \int_0^t \mathcal D(s)\dif s \right). \label{3.18}
\end{align}
\end{lemma}

\begin{proof}
Taking derivative to the equations \eqref{2.1} with respect to $t$ and $x$ once, respectively, we get
\begin{align}\label{3.19}
\begin{cases}
v_{ttx}=u_{txx},\\
u_{ttx}+p(v)_{txx}=S_{txx},\\
\tau \left(\frac{S_t}{v}\right)_{tx}+\left(\tau \epsilon b(x)\frac{S_x}{v} \right)_{tx}+S_{tx}=\left(\frac{\mu}{v} u_x\right)_{tx}.
\end{cases}
\end{align}
Multiplying the  equation $\eqref{3.19}_2$ by $u_{tx}$ and integrating the result, we get
\begin{align*}
\frac{\dif}{\dif t} \int_0^1 \frac{1}{2} u_{tx}^2 \dif x =\int_0^1 (S_{txx}-p(v)_{txx})u_{tx}\dif x =-\int_0^1 (S_{tx}-p(v)_{tx}) u_{txx}\dif x,
\end{align*}
where we have used the condition $u_{tt}(t,0)=u_{tt}(t,1)=0$ (formally) which is equivalent to $$(S_{tx}-p(v)_{tx})(t,0)=(S_{tx}-p(v)_{tx})(t,1)=0.$$
Note that
\begin{align*}
&\int_0^1 p(v)_{tx} u_{txx}\dif x =\int_0^1 (p^\prime(v) v_{tx} +p^{\prime\prime}(v) v_x v_t) u_{txx}\dif x\\
&=\frac{\dif}{\dif t} \int_0^1 \frac{1}{2} p^\prime(v)v_{tx}^2 \dif x-\frac{1}{2}\int_0^1 p^{\prime\prime}(v) v_t v_{tx}^2 \dif x+\frac{\dif}{\dif t} \int_0^1 p^{\prime\prime}(v) v_x v_t u_{xx}\dif x-\int_0^1 (p^{\prime\prime}(v) v_x v_t)_t u_{xx}\dif x\\
&\le \frac{\dif}{\dif t} \int_0^1 (\frac{1}{2} p^\prime(v)v_{tx}^2+p^{\prime\prime}(v) v_x v_t u_{xx}) \dif x+C E^\frac{1}{2}(t) \mathcal D(t).
\end{align*}
Thus, we get
\begin{align}\label{3.20}
\frac{\dif}{\dif t} \int_0^1 (\frac{1}{2} u_{tx}^2-\frac{1}{2} p^\prime(v)v_{tx}^2) \dif x 
\le \frac{\dif}{\dif t} \int_0^1 p^{\prime\prime}(v) v_x v_t u_{xx}\dif x- \int_0^1 S_{tx} u_{txx}\dif x+C E^\frac{1}{2}(t) \mathcal D(t).
\end{align}
On the other hand, by multiplying the equation $\eqref{3.19}_3$ by $\frac{v}{\mu} S_{tx}$, and integrating the result, we get
\begin{align*}
\int_0^1 \tau \left(\frac{S_t}{v}\right)_{tx} \frac{v}{\mu} S_{tx} \dif x+\int_0^1\left( \tau\epsilon b(x) \frac{S_x}{v}\right)_{tx} \frac{v}{\mu} S_{tx}\dif x+\int_0^1 \frac{v}{\mu} S_{tx}^2 \dif x 
=\int_0^1 \left(\frac{\mu u_x}{v}\right)_{tx} \frac{v}{\mu} S_{tx} \dif x.
\end{align*}
We estimate each term of the above equation separately.  For the left-hand side of the above equation, we have
\begin{align*}
&\int_0^1 \tau \left(\frac{S_t}{v}\right)_{tx} \frac{v}{\mu} S_{tx} \dif x\\
&=\int_0^1 \tau \left( \frac{1}{v} S_{ttx}-\frac{v_x}{v^2}  S_{tt}-\frac{v_t}{v^2}S_{tx}-\frac{S_t}{v^2}v_{tx}+\frac{2}{v^3}v_t S_t v_x\right) \frac{v}{\mu} S_{tx} \dif x \ge \frac{\dif}{\dif t} \int_0^1 \frac{\tau}{2\mu}  S_{tx}^2 \dif x-C E^\frac{1}{2}(t) \mathcal D(t)
\end{align*}
and
\begin{align*}
&\int_0^1 \tau \epsilon  \left(\frac{b(x)S_x}{v}\right)_{tx} \frac{v}{\mu} S_{tx}\dif x\\
&=\int_0^1 \tau \epsilon b(x) \left( \frac{1}{v} S_{txx}+\frac{2}{v}S_{tx}-\frac{2v_t}{v^2}S_x-\frac{v_x}{v^2}  S_{tx}-\frac{v_t}{v^2}S_{xx}+\frac{2v_tv_x}{v^3} S_x-\frac{v_{tx}}{v^2} S_x\right) \frac{v}{\mu} S_{tx} \dif x \\
&\ge \frac{1}{2\mu} \tau \epsilon b(x) S_{tx}^2\big|_0^1+\int_0^1 \frac{\tau \epsilon}{\mu} S_{tx}^2 \dif x-C E^\frac{1}{2}(t) \mathcal D(t).
\end{align*}
For the right-hand side of the above equation, we have
\begin{align*}
\int_0^1 \left(\frac{\mu}{v} u_x\right)_{tx} \frac{v}{\mu} S_{tx}\dif x
&=\int_0^1 \left(\frac{2}{v^3} v_t v_x u_x-\frac{1}{v^2} v_{tx}u_x- \frac{1}{v^2} v_t u_{xx}-\frac{1}{v^2}v_x u_{xt}+\frac{1}{v} u_{txx}\right) v S_{tx} \dif x\\
&\le \int_0^1 u_{txx} S_{tx} \dif x +C E^\frac{1}{2}(t) \mathcal D(t).
\end{align*}

Therefore, one derive that
\begin{align}
\frac{\dif}{\dif t} \int_0^1 \frac{\tau}{2\mu} S_{tx}^2 \dif x-\int_0^1 u_{txx}S_{tx}\dif x+\frac{1}{4\mu} \int_0^1 S_{tx}^2\dif x \le C E^\frac{1}{2}(t) \mathcal D(t). \label{3.21}
\end{align}
Combining \eqref{3.20} and \eqref{3.21}, using Lemma \ref{Le3.5} and integrating the result over $(0, t)$, we get the desired result.

\end{proof}

\begin{lemma}\label{Le3.9}
There exists a constant $C$ such that
\begin{align}
\int_0^1 (-\frac{1}{2} p^\prime(v)v_{xx}^2+\frac{1}{2}u_{xx}^2+\frac{\tau}{2\mu}S_{xx}^2 )\dif x +\frac{1}{4\mu}\int_0^t\int_0^1 S_{xx}^2 \dif x+\frac{\tau\epsilon}{2\mu}(S_{xx}^2(t,1)+S_{xx}^2(t,0)) \nonumber\\
\le C \left(E_0+E^\frac{3}{2}(t))+\int_0^t u_{tx}u_{xx}\big|_0^1 +E^\frac{1}{2}(t) \int_0^t \mathcal D(s)\dif s  \right). \label{3.22}
\end{align}
\end{lemma}
\begin{proof}
Taking derivative to the equation \eqref{2.1} with respect to $x$ twice, we get
\begin{align}\label{3.23}
\begin{cases}
v_{txx}=u_{xxx},\\
u_{txx}+p(v)_{xxx}=S_{xxx},\\
\tau \left(\frac{S_t}{v}\right)_{xx}+\tau \epsilon \left(\frac{b(x)S_x}{v}\right)_{xx}+S_{xx}=\left(\frac{\mu u_x}{v}\right)_{xx}.
\end{cases}
\end{align}
Multiplying the  equation $\eqref{3.23}_2$ by $u_{xx}$, one obtain
\begin{align*}
\frac{\dif}{\dif t} \int_0^1 \frac{1}{2} u_{xx}^2 \dif x 
&=-\int_0^1 (p(v)_{xxx}-S_{xxx})u_{xx}\dif x\\
&=u_{tx} u_{xx}\big|_0^1 +\int_0^1 p(v)_{xx} u_{xxx}\dif x-\int_0^1 S_{xx} u_{xxx}\dif x.
\end{align*}
Note that, by using the equation $\eqref{3.23}_1$, we have
\begin{align*}
&\int_0^1 p(v)_{xx}u_{xxx}\dif x=\int_0^1 (p^\prime(v) v_{xx}+p^{\prime\prime}(v) v_x^2)v_{txx} \dif x\\
&=\frac{\dif}{\dif t} \int_0^1 \frac{1}{2}p^\prime(v) v_{xx}^2 \dif x-\frac{1}{2} \int_0^1 p^{\prime\prime}(v)v_t v_{xx}^2 \dif x\\
&\qquad\qquad\qquad+\frac{\dif}{\dif t} \int_0^1 p^{\prime\prime}(v) v_x^2 v_{xx} \dif x-\int_0^1 (p^{\prime\prime\prime}(v)v_t v_{x}^2+2p^{\prime\prime}(v) v_x v_{xt} )v_{xx}\dif x.
\end{align*}
Therefore, we get
\begin{align}\label{3.24}
\frac{\dif}{\dif t} \int_0^1 (\frac{1}{2} u_{xx}^2-\frac{1}{2} p^\prime(v) v_{xx}^2-p^{\prime\prime}(v)v_x^2v_{xx}) \dif x+\int_0^1 S_{xx} u_{xxx}\dif x\le u_{tx}u_{xx}\big|_0^1+C E^\frac{1}{2}(t) \mathcal D(t).
\end{align}
On the other hand, multiplying the  equation $\eqref{3.23}_3$ by $\frac{v}{\mu} S_{xx}$ and integrating the result, we have
\begin{align*}
\int_0^1 \tau \left(\frac{S_t}{v}\right)_{xx} \frac{v}{\mu} S_{xx} \dif x+\int_0^1\left( \tau\epsilon b(x) \frac{S_x}{v}\right)_{xx} \frac{v}{\mu} S_{xx}\dif x+\int_0^1 \frac{v}{\mu} S_{xx}^2 \dif x 
=\int_0^1 \left(\frac{\mu u_x}{v}\right)_{xx} \frac{v}{\mu} S_{xx} \dif x.
\end{align*}
Firstly, we  estimate the first two terms of the above equation as follows.
\begin{align*}
&\int_0^1 \tau \left(\frac{S_t}{v}\right)_{xx} \frac{v}{\mu} S_{xx}\dif x\\
&=\int_0^1 \tau \left(\frac{S_{txx}}{v}-\frac{2v_x}{v^2} S_{tx}+\frac{2v_x^2}{v^3} S_t-\frac{v_{xx}}{v^2} S_t\right) \frac{v}{\mu} S_{xx}\dif x
\ge \frac{\dif}{\dif t} \int_0^1 \frac{\tau}{2\mu} S_{xx}^2 \dif x -CE^\frac{1}{2}(t) \mathcal D(t)
\end{align*}
and
\begin{align*}
&\int_0^1 \tau \epsilon \left(\frac{b(x)S_x}{v}\right)_{xx} \frac{v}{\mu} S_{xx}\dif x \\
&=\int_0^1 \tau \epsilon \left( \frac{b(x)}{v}S_{xxx}+\frac{4}{v} S_{xx}-\frac{2 b(x) v_x}{v^2}S_{xx}-\frac{4v_x}{v^2}S_x-\frac{b(x)v_{xx}}{v^2}S_x+\frac{2 b(x)v_x^2}{v^3} S_x\right)\frac{v}{\mu} S_{xx}\dif x\\
&\ge \frac{\tau\epsilon}{\mu} b(x) \frac{1}{2} S_{xx}^2 \big|_0^1+\frac{3\tau \epsilon}{\mu}\int_0^1 S_{xx}^2 \dif x -C E^\frac{1}{2}(t)\mathcal D(t) \ge \frac{\tau\epsilon}{2\mu}(S_{xx}^2(t,1)+S_{xx}^2(t,0)) -C E^\frac{1}{2}(t)\mathcal D(t).
\end{align*}
Secondly, we obtain
\begin{align*}
&\int_0^1 \left(\frac{\mu u_x}{v}\right)_{xx} \frac{v}{\mu} S_{xx}\dif x \\
&=\int_0^1 \left(\frac{1}{v} u_{xxx}-\frac{2v_x}{v^2} u_{xx}+\frac{2v_x^2}{v^3} u_x-\frac{v_{xx}}{v^2}u_x\right) v S_{xx}\dif x \le \int_0^1 u_{xxx} S_{xx}\dif x +C E^\frac{1}{2}(t)\mathcal D(t).
\end{align*}

Therefore, we get
\begin{align}
&\frac{\dif}{\dif t} \int_0^1 \frac{\tau}{2\mu} S_{xx}^2 \dif x+\frac{1}{4\mu} \int_0^1 S_{xx}^2\dif x +\frac{\tau\epsilon}{2\mu}(S_{xx}^2(t,1)+S_{xx}^2(t,0))-\int_0^1 u_{xxx}S_{xx} \dif x \nonumber\\
&\le  C E^\frac{1}{2}(t)\mathcal D(t) . \label{3.25}
\end{align} 
Combining the equations \eqref{3.24} and \eqref{3.25} and integrating the result over $(0, t)$, we get the desired result.
\end{proof}

Now, we prove the following key lemma which deal with the boundary term $u_{tx}u_{xx}\big|_0^1$. 

\begin{lemma}\label{Le3.10}
There exists a constant $C$ such that
\begin{align}
&\int_0^t \frac{1}{2} (u_{xx}^2(t,1)+u_{xx}^2(t,0)+ u_{tx}^2(t, 1)+u_{tx}^2(t, 0) ) \dif t \nonumber\\
 &\le C\int_0^1 (u_{tx}^2+u_{tt}^2+v_{tx}^2+\epsilon  v_{xx}^2 )\dif x
+C  \int_0^t\int_0^1 (u_{tx}^2+u_{xx}^2+u_{tt}^2 )\dif x \dif t \nonumber\\
&+C \epsilon \int_0^t\int_0^1 (v_{xx}^2+u_{xx}^2)\dif x \dif t +\frac{1}{8\mu} \int_0^t\int_0^1 S_{xx}^2 \dif x \dif t
+\frac{\tau \epsilon}{4\mu} \int_0^t (S_{xx}^2(t,1)+S_{xx}^2(t, 0)) \dif t +CE^\frac{1}{2}\int_0^t \mathcal D(s)\dif s .
\end{align}

\end{lemma}

\begin{proof}
 For simplicity, denote $\bar b(x):=-b(x)$ with $\bar b(0)=1, \bar b (1)=-1$. Multiplying the  equation $\eqref{3.23}_3$ by $\frac{v}{\mu} \bar b(x) u_{xx}$ and integrating the result, we get
\begin{align}
\int_0^1 \tau \left(\frac{S_t}{v}\right)_{xx} \frac{v}{\mu} \bar b(x) u_{xx} \dif x+\int_0^1 S_{xx}\frac{v}{\mu}\bar b(x) u_{xx} \dif x 
-\int_0^1 \left(\frac{\mu u_x}{v}\right)_{xx} \frac{v}{\mu} \bar b(x) u_{xx} \dif x \nonumber\\
=-\int_0^1\left( \tau\epsilon b(x) \frac{S_x}{v}\right)_{xx} \frac{v}{\mu} \bar b(x) u_{xx}\dif x . \label{3.26}
\end{align}
We estimate each term of the above equation, respectively. Firstly, we have
\begin{align*}
\int_0^1 \tau\left(\frac{S_t}{v}\right)_{xx} \frac{v}{\mu} \bar b(x) u_{xx} \dif x
=\int_0^1 \tau \left(\frac{S_{txx}}{v}-\frac{2v_x}{v^2} S_{tx}+\frac{2v_x^2}{v^3} S_t-\frac{v_{xx}}{v^2} S_t\right)\frac{v}{\mu} \bar b(x) u_{xx}\dif x\\
\ge \frac{\tau}{\mu} \frac{\dif}{\dif t}\int_0^1 S_{xx} \bar b(x) u_{xx}\dif x -\frac{\tau}{\mu}\int_0^1 (u_{tx}+p(v)_{xx}) \bar b(x) u_{txx}\dif x -C E^\frac{1}{2}(t) \mathcal D(t)
\end{align*}
where
\begin{align*}
-\frac{\tau}{\mu} \int_0^1 u_{tx} \bar b(x) u_{txx}\dif x=\frac{\tau}{\mu} (u_{tx}^2(t,1)+u_{tx}^2(t, 0)) -\frac{2\tau}{\mu}\int_0^1  u_{tx}^2\dif x 
\end{align*}
and
\begin{align*}
&-\frac{\tau}{\mu} \int_0^1 p(v)_{xx} \bar b(x) u_{txx}\dif x=-\frac{\dif}{\dif t} \int_0^1 \frac{\tau}{\mu}  p(v)_{xx}\bar b(x) u_{xx} \dif x+\frac{\tau}{\mu}\int_0^1 p(v)_{xxt} \bar b(x)u_{xx}\dif x\\
&\ge-\frac{\dif}{\dif t}\int_0^1 \frac{\tau}{\mu} p_{xx}(v) \bar b u_{xx}\dif x +\frac{\tau}{\mu} \int_0^1 p^\prime(v)u_{xx}^2\dif x +\frac{\tau}{\mu} p^\prime(v) u_{xx}^2 \bar b(x)\big|_0^1-CE^\frac{1}{2}(t)\mathcal D(t)\\
 &\ge-\frac{\dif}{\dif t}\int_0^1 \frac{\tau}{\mu} p_{xx}(v) \bar b u_{xx}\dif x +\frac{\tau}{\mu} \int_0^1 p^\prime(v)u_{xx}^2\dif x-CE^\frac{1}{2}(t)\mathcal D(t),
\end{align*} 
since $p^\prime(v)<0$. 

For the last term on the left-hand side of equation \eqref{3.26}, we get
\begin{align*}
&-\int_0^1 \left(\frac{\mu u_x}{v}\right)_{xx} \frac{v}{\mu} \bar b(x) u_{xx}\dif x\\
&=-\int_0^1\left(\frac{1}{v} u_{xxx}-\frac{2v_x}{v^2} u_{xx}+\frac{2v_x^2}{v^3} u_x-\frac{v_{xx}}{v^2}u_x\right)v \bar b(x)u_{xx}\dif x\\
&\ge -\frac{1}{2}\bar b(x) u_{xx}^2 \big|_0^1-\int_0^1 u_{xx}^2\dif x-CE^\frac{1}{2}(t)\mathcal D(t)\\
&\ge \frac{1}{2} (u_{xx}^2(t,1)+u_{xx}^2(t, 0))-\int_0^1 u_{xx}^2 \dif x-CE^\frac{1}{2}(t)\mathcal D(t).
\end{align*}

For the right-hand side of the equation \eqref{3.26}, we have
\begin{align*}
&\int_0^1 \tau \epsilon \left(\frac{b(x)S_x}{v}\right)_{xx} \frac{v}{\mu} \bar b(x)u_{xx}\dif x \\
&=\int_0^1 \tau \epsilon \left( \frac{b(x)}{v}S_{xxx}+\frac{4}{v} S_{xx}-\frac{2 b(x) v_x}{v^2}S_{xx}-\frac{4v_x}{v^2}S_x-\frac{b(x)v_{xx}}{v^2}S_x+\frac{2 b(x)v_x^2}{v^3} S_x\right)\frac{v}{\mu} \bar b(x)u_{xx}\dif x\\
&\le \frac{\tau \epsilon}{\mu} \int_0^1 b(x) S_{xxx} \bar b(x) u_{xx} \dif x +\frac{4\tau \epsilon}{\mu} \int_0^1 S_{xx} u_{xx} \dif x +CE^\frac{1}{2}(t)\mathcal D(t).
\end{align*}
In view of  equations $\eqref{2.1}_1$-$\eqref{2.1}_2$, we have
\begin{align*}
&\int_0^1 b(x) S_{xxx} \bar b(x) u_{xx}\dif x=-\int_0^1 b^2(x)(u_{txx}+p(v)_{xxx})u_{xx}\dif x\\
&\le -\frac{\dif}{\dif t}\int_0^1 \frac{1}{2}b^2(x) u_{xx}^2\dif x-\int_0^1 b^2(x)p^\prime(v) v_{xxx}u_{xx} \dif x+CE^\frac{1}{2}(t)\mathcal D(t)\\
&\le -\frac{\dif}{\dif t} \int_0^1 \frac{1}{2} b^2(x)u_{xx}^2 \dif x-\left[b^2(x)p^\prime(v) v_{xx}u_{xx} \right]\big|_0^1+\int_0^1 (b^2(x)p^\prime(v) u_{xx})_x v_{xx}\dif x +CE^\frac{1}{2}(t)\mathcal D(t)\\
&\le \frac{\dif}{\dif t} \int_0^1 (-\frac{1}{2} b^2 u_{xx}^2+\frac{1}{2} b^2(x)p^\prime(v) v_{xx}^2)\dif x+\eta[ v_{xx}^2(t,1)+v_{xx}^2(t,0)]+C(\eta) [u_{xx}^2(t,1)+u_{xx}^2(t,0)]\\
&\qquad\qquad\qquad\qquad\qquad\qquad\qquad\qquad\qquad\qquad+4\int_0^1 b(x)p^\prime(v) u_{xx}v_{xx}\dif x+CE^\frac{1}{2}(t)\mathcal D(t).
\end{align*}
where $\eta$ is sufficiently small and to be chosen later.
 
Combining the above result, we get
\begin{align*}
&\frac{1}{2} (u_{xx}^2(t,1)+u_{xx}^2(t,0))+\frac{\tau}{\mu} (u_{tx}^2 (t,1)+u_{tx}^2(t,0))\\
&\le \frac{\dif}{\dif t} \int_0^1 \left(-\frac{\tau}{\mu} \bar b(x) (S_{xx}-p_{xx}(v))u_{xx}\dif x+\frac{\tau \epsilon}{2\mu} b^2(x)(u_{xx}^2-p^\prime(v) v_{xx}^2)\right)\dif x\\
&+ \int_0^1\left( \frac{\tau}{\mu}(2 u_{tx}^2-p^\prime(v) u_{xx}^2 )+ \frac{v}{\mu} \bar b(x) S_{xx} u_{xx}+ u_{xx}^2+\frac{4\tau\epsilon}{\mu}  (S_{xx}+b(x)p^\prime(v) v_{xx}) u_{xx}\right) \dif x+C E^\frac{1}{2}(t)\mathcal D(t)\\
&+\frac{\tau\epsilon}{\mu} \eta[ v_{xx}^2(t,1)+v_{xx}^2(t,0)]+\frac{\tau\epsilon}{\mu} C(\eta) [u_{xx}^2(t,1)+u_{xx}^2(t,0)] .
\end{align*}
On the other hand, from the momentum equation $\eqref{2.1}_2$, we have
\begin{align*}
v_{xx}^2(t,1)+v_{xx}^2(t,0)\le C (S_{xx}^2(t,1)+S_{xx}^2(t,0)+u_{tx}^2(t,1)+u_{tx}^2(t,0))+C E^\frac{1}{2}(t)\mathcal D(t).
\end{align*}
Therefore, by choosing $\eta=\frac{1}{4C} $, $\epsilon$ sufficiently small and  integrating over $(0, t)$, we get
\begin{align}
&\int_0^t \frac{1}{2} (u_{xx}^2(t,1)+u_{xx}^2(t,0)) \dif t \nonumber\\
&\le \int_0^1 (v_{tx}^2+\epsilon  v_{xx}^2 +C u_{tx}^2)\dif x+C  \int_0^t\int_0^1 (u_{tx}^2+u_{xx}^2 )\dif x \dif t+CE^\frac{1}{2}\int_0^t \mathcal D(s)\dif s \nonumber\\
&+C \epsilon \int_0^t\int_0^1 v_{xx}^2\dif x \dif t +\frac{1}{8\mu} \int_0^t\int_0^1 S_{xx}^2 \dif x \dif t
+\frac{\tau \epsilon}{4\mu} \int_0^t (S_{xx}^2(t,1)+S_{xx}^2(t, 0)) \dif t . \label{3.28}
\end{align}

Next, we we give estimates of $u_{tx}$ on the boundary. Multiplying  $\eqref{3.19}_3$ by $\frac{v}{\mu} \bar b(x) u_{tx}$ and integrating the result, we get
\begin{align}
\int_0^1 \tau \left(\frac{S_t}{v}\right)_{tx} \frac{v}{\mu} \bar b(x) u_{tx} \dif x+\int_0^1 S_{tx}\frac{v}{\mu}\bar b(x) u_{tx} \dif x 
-\int_0^1 \left(\frac{\mu u_x}{v}\right)_{tx} \frac{v}{\mu} \bar b(x) u_{tx} \dif x \nonumber\\
=-\int_0^1\left( \tau\epsilon b(x) \frac{S_x}{v}\right)_{tx} \frac{v}{\mu} \bar b(x) u_{tx}\dif x . \label{3.29}
\end{align}
We estimate each term of the above equation, respectively. Firstly, we have
\begin{align*}
\int_0^1 \tau\left(\frac{S_t}{v}\right)_{tx} \frac{v}{\mu} \bar b(x) u_{tx} \dif x
=\int_0^1 \tau \left(\frac{S_{ttx}}{v}-\frac{2v_x}{v^2} S_{tt}+\frac{2v_xv_t}{v^3} S_t-\frac{v_{tx}}{v^2} S_t\right)\frac{v}{\mu} \bar b(x) u_{tx}\dif x\\
\ge \frac{\tau}{\mu} \frac{\dif}{\dif t}\int_0^1 S_{tx} \bar b(x) u_{tx}\dif x -\frac{\tau}{\mu}\int_0^1 (u_{tt}+p(v)_{tx}) \bar b(x) u_{ttx}\dif x -C E^\frac{1}{2}(t) \mathcal D(t),
\end{align*}
where
\begin{align*}
-\frac{\tau}{\mu} \int_0^1 u_{tt} \bar b(x) u_{ttx}\dif x=-\frac{2\tau}{\mu}\int_0^1  u_{tt}^2\dif x
\end{align*}
and
\begin{align*}
&-\frac{\tau}{\mu} \int_0^1 p(v)_{tx} \bar b(x) u_{ttx}\dif x=-\frac{\dif}{\dif t} \int_0^1 \frac{\tau}{\mu}  p(v)_{tx}\bar b(x) u_{tx} \dif x+\frac{\tau}{\mu}\int_0^1 p(v)_{ttx} \bar b(x)u_{tx}\dif x\\
&\ge-\frac{\dif}{\dif t}\int_0^1 \frac{\tau}{\mu} p_{tx}(v) \bar b u_{tx}\dif x +\frac{\tau}{\mu} \int_0^1 p^\prime(v)u_{tx}^2\dif x +\frac{\tau}{\mu} p^\prime(v) u_{tx}^2 \bar b(x)\big|_0^1-CE^\frac{1}{2}(t)\mathcal D(t)\\
 &\ge-\frac{\dif}{\dif t}\int_0^1 \frac{\tau}{\mu} p_{tx}(v) \bar b u_{tx}\dif x +\frac{\tau}{\mu} \int_0^1 p^\prime(v)u_{tx}^2\dif x-CE^\frac{1}{2}(t)\mathcal D(t),
\end{align*} 
since $p^\prime(v)<0$. 

For the last term on the left-hand side of equation \eqref{3.28}, we get
\begin{align*}
&-\int_0^1 \left(\frac{\mu u_x}{v}\right)_{tx} \frac{v}{\mu} \bar b(x) u_{tx}\dif x\\
&=-\int_0^1\left(\frac{1}{v} u_{txx}-\frac{v_x}{v^2} u_{tx}-\frac{v_t}{v^2}u_{xx}+\frac{2v_x v_t}{v^3} u_x-\frac{v_{tx}}{v^2}u_x\right)v \bar b(x)u_{tx}\dif x\\
&\ge -\frac{1}{2}\bar b(x) u_{tx}^2 \big|_0^1-\int_0^1 u_{tx}^2\dif x-CE^\frac{1}{2}(t)\mathcal D(t)\\
&\ge \frac{1}{2} (u_{tx}^2(t,1)+u_{tx}^2(t, 0))-\int_0^1 u_{tx}^2 \dif x-CE^\frac{1}{2}(t)\mathcal D(t).
\end{align*}

For the right-hand side of the equation \eqref{3.28}, we have
\begin{align*}
&\int_0^1 \tau \epsilon \left(\frac{b(x)S_x}{v}\right)_{tx} \frac{v}{\mu} \bar b(x)u_{tx}\dif x \\
&=\int_0^1 \tau \epsilon \left( \frac{b(x)}{v}S_{txx}-\frac{ b(x) v_t}{v^2}S_{xx}-\frac{b(x) v_x}{v^2} S_{tx}-\frac{b(x)v_{tx}}{v^2}S_x+\frac{2 b(x)v_x v_t}{v^3} S_x\right)\frac{v}{\mu} \bar b(x)u_{tx}\dif x\\
&\le \frac{\tau \epsilon}{\mu} \int_0^1 b(x) S_{txx} \bar b(x) u_{tx} \dif x  +CE^\frac{1}{2}(t)\mathcal D(t).
\end{align*}
Using the  equations $\eqref{2.1}_1-\eqref{2.1}_2$, we have
\begin{align*}
&\int_0^1 b(x) S_{txx} \bar b(x) u_{tx}\dif x=-\int_0^1 b^2(x)(u_{ttx}+p(v)_{txx})u_{tx}\dif x\\
&\le -\frac{\dif}{\dif t}\int_0^1 \frac{1}{2}b^2(x) u_{tx}^2\dif x-\int_0^1 b^2(x)p^\prime(v) v_{txx}u_{tx} \dif x+CE^\frac{1}{2}(t)\mathcal D(t)\\
&\le -\frac{\dif}{\dif t} \int_0^1 \frac{1}{2} b^2(x)u_{tx}^2 \dif x-\left[b^2(x)p^\prime(v) u_{xx}u_{tx} \right]\big|_0^1+\int_0^1 (b^2(x)p^\prime(v) u_{tx})_x u_{xx}\dif x +CE^\frac{1}{2}(t)\mathcal D(t)\\
&\le \frac{\dif}{\dif t} \int_0^1 (-\frac{1}{2} b^2 u_{tx}^2+\frac{1}{2} b^2(x)p^\prime(v) u_{xx}^2)\dif x+C[ u_{xx}^2(t,1)+u_{xx}^2(t,0)+u_{tx}^2(t,1)+u_{tx}^2(t,0)]\\
&\qquad\qquad\qquad\qquad\qquad\qquad\qquad\qquad\qquad\qquad+4\int_0^1 b(x)p^\prime(v) u_{xx}u_{tx}\dif x+CE^\frac{1}{2}(t)\mathcal D(t).
\end{align*}

Combining the above result, we get
\begin{align*}
&\frac{1}{2} (u_{tx}^2(t,1)+u_{tx}^2(t,0))\le \frac{\dif}{\dif t} \int_0^1 \left(-\frac{\tau}{\mu} \bar b(x) (S_{tx}-p_{tx}(v))u_{tx}\dif x+\frac{\tau \epsilon}{2\mu} b^2(x)(u_{tx}^2-p^\prime(v) u_{xx}^2)\right)\dif x\\
&+C\epsilon\int_0^1 u_{xx}^2 \dif x+ \int_0^1\left( \frac{\tau}{\mu}(2 u_{tt}^2-p^\prime(v) u_{tx}^2 )+ \frac{v}{\mu} \bar b(x) S_{tx} u_{tx}+ u_{tx}^2\right) \dif x+C E^\frac{1}{2}(t)\mathcal D(t)\\
&+\frac{\tau\epsilon}{\mu} [ u_{xx}^2(t,1)+u_{xx}^2(t,0)].
\end{align*}
Integrating the above result over $(0, t)$, one obtain
\begin{align}
\int_0^t \frac{1}{2}( u_{tx}^2(t, 1)+u_{tx}^2(t, 0) )\dif t 
\le C(E_0+CE^\frac{1}{2}\int_0^t \mathcal D(s)\dif s+\int_0^1 (u_{tt}^2+u_{tx}^2+v_{tx}^2+\epsilon u_{xx}^2)\dif x \nonumber\\
+ \epsilon \int_0^t \int_0^1 u_{xx}^2 \dif x \dif t + \int_0^t\int_0^1 (u_{tt}^2+u_{tx}^2)\dif x \dif t 
+\epsilon \int_0^t ( u_{xx}^2(t,1)+u_{xx}^2(t,0)) \dif t) . \label{3.30}
\end{align}
Combining \eqref{3.28} and \eqref{3.30},  we get the desired result.
\end{proof}

On the other hand, using the equations $\eqref{2.1}_1$, $\eqref{2.1}_2$ and $\eqref{2.1}_3$ , we can easily get the second-order dissipation of $(v, u)$.
\begin{lemma}\label{Le3.11}
There exists a constant $C$ such that
\begin{align}
&\int_0^t\int_0^1 (u_{xx}^2+u_{tx}^2+u_{tt}^2+v_{tt}^2+v_{tx}^2)\dif x \dif t \nonumber \\
&\le C ( \int_0^t \int_0^1\tau^2 (S_{tx}^2+S_{tt}^2)\dif x\dif t+\epsilon \int_0^t\int_0^1 S_{xx}^2\dif x\dif t +E_0+E^\frac{1}{2}(t) \int_0^t \mathcal D(s)\dif s
 \label{3.31}
\end{align}
and
\begin{align}\label{3.32}
\int_0^t\int_0^1 v_{xx}^2\dif x \dif t \le C\int_0^t\int_0^1 (u_{tx}^2+S_{xx}^2)\dif x+CE (t)\int_0^t \mathcal D(s)\dif s.
\end{align}
\end{lemma}

Therefore, combining Lemmas \ref{Le3.5}-\ref{Le3.11}, the proof of Proposition \ref{Prop3.1} is finished.

\section{Passing to the Limit and  proof of Main Theorems}
{\bf Proof of Theorem \ref{th1.1}:} According to Proposition \ref{Prop3.1},  the local solution $(v^\epsilon, u^\epsilon, S^\epsilon)$ can be extended to $[0, \infty)$ by usual continuation methods. Thus, for fixed $\epsilon$, we get a global solution $(v^\epsilon, u^\epsilon, S^\epsilon)$ to system \eqref{2.1}-\eqref{2.3} satisfying 
\begin{align}\label{4.1}
\sup_{0\le t <\infty}\|(v^\epsilon-1, u^\epsilon, \sqrt{\tau}S^\epsilon)(t, \cdot)\|_{H^2}^2+ \int_0^\infty \left( \|(v^\epsilon_x, u^\epsilon_x)\|_{H^1}^2+ \|S^\epsilon\|_{H^2}^2 \right) \dif t 
\le C E_0,
\end{align}
where $C$ is a constant independent of $\epsilon$ and $\tau$.
Therefore, the uniform bounds of $(v^\epsilon, u^\epsilon, \sqrt{\tau} S^\epsilon)$ in $L^\infty([0, \infty), H^2(\Omega)$  implies  that there exists 
$(v, u, \sqrt{\tau} S)\in L^\infty([0, \infty), H^2(\Omega)$ such  that 
\begin{align*}
(v^\epsilon, u^\epsilon, \sqrt{\tau} S^\epsilon) \rightharpoonup (v, u, \sqrt{\tau} S)\quad \mathrm{weak-}\ast \, \mathrm{in}\quad L^\infty([0, \infty), H^2(\Omega).
\end{align*}
On the other hand, since $(\partial_tv^\epsilon, \partial_t u^\epsilon, \partial_t S^\epsilon)$ are bounded in $L^2(0, T; H^1)$ for any $T>0$, we have
$(v, u, S)\in C([0,T], H^1)$. Moreover, by classical compactness theorem, for any $\delta_0>0$, $(v^\epsilon, u^\epsilon, S^\epsilon)$ are relatively compact in $C([0, T], H^{2-\delta_0})$. As a consequence, as $\epsilon\rightarrow 0$ and up to subsequences, 
\begin{align*}
(v^\epsilon, u^\epsilon, S^\epsilon)\rightarrow (v, u, S), \quad \mathrm{strongly}\quad \mathrm{in}\quad C([0, T], H^{2-\delta_0}(\Omega)).
\end{align*}
Passing to the limit in \eqref{2.1} and \eqref{4.1}, the proof of Theorem \ref{th1.1} is finished.\\

{\bf Proof of Theorem \ref{th1.2}:}  Let  $(v^\tau, u^\tau, S^\tau)$ be the global solutions obtained in Theorem \ref{th1.1}, then we have
\begin{align}
\sup_{0\le t <\infty}\|(v^\tau-1, u^\tau, \sqrt{\tau}S^\tau)(t, \cdot)\|_{H^2}^2+ \int_0^\infty \left( \|(v^\tau_x, u^\tau_x)\|_{H^1}^2+ \|S^\tau\|_{H^2}^2 \right) \dif t \le C E_0,
\end{align}
where $C$ is a constant independent of $\tau$. Thus, there exists $(v^0, u^0)\in L^\infty((0, \infty), H^2(\Omega))$ and $S^0\in L^2((0, \infty), H^2(\Omega))$ such that
\begin{align*}
(v^\tau, u^\tau) \rightharpoonup (v^0,  u^0)\quad \mathrm{weak-}\ast \, \mathrm{in}\quad L^\infty((0, \infty), H^2(\Omega)),\\
S^\tau  \rightharpoonup S^0\quad \mathrm{weakly} \quad \mathrm{in}\quad L^2((0, \infty), H^2(\Omega))
\end{align*}
For any $T>0$, it is easy to see that both $\partial_t v^\tau$ and $\partial_t u^\tau$ are bounded in $L^2((0,T), H^1(\Omega))$. Therefore, $(v^0, u^0)\in C([0, T], H^1(\Omega))$. Furthermore, by classical compactness theorem, for any $\delta_0>0$, $(v^\tau, u^\tau)$ are relatively compact in $C([0, T], H^{2-\delta_0})$. As a consequence, as $\tau\rightarrow 0$ and up to subsequences, 
\begin{align*}
(v^\tau, u^\tau)\rightarrow (v^0, u^0), \quad \mathrm{strongly}\quad \mathrm{in}\quad C([0, T], H^{2-\delta_0}(\Omega)).
\end{align*}
On the other hand, the uniform boundedness of $\sqrt{\tau}S^\tau$ yields $\tau S^\tau \rightarrow 0$ in $L^\infty((0,\infty), H^2(\Omega))$ as $\tau \rightarrow 0$, which leads to 
$\tau \partial_t S^\tau \rightarrow 0$ in $D^\prime((0,\infty), H^2(\Omega))$ as $\tau\rightarrow 0$. Therefore, 
\begin{align}
S^0=\frac{\mu (u^0)_x}{v^0} \quad \mathrm{a.e.} \quad (0,\infty)\times \Omega.
\end{align}
Then, passing to the limit in \eqref{1.5}, the proof of Theorem \ref{th1.2} is finished.

{\bf Acknowledgement:} Yuxi Hu's Research is supported by the Fundamental Research Funds for the Central Universities (No. 2023ZKPYLX01).

{\bf Data availability statement}: There is no data used in this paper.

\end{document}